\documentclass[11pt,letterpaper]{article}

\usepackage{enumerate}
\usepackage[pdftex]{color}
\usepackage{amsmath,amsfonts,amssymb}
\usepackage[all]{xy}
\usepackage{graphicx}
\usepackage[margin=2.5cm]{geometry} 
\usepackage{fancyhdr}
\usepackage{hyperref}
\usepackage{multirow}
\usepackage[hyperref]{ntheorem}

\def\ser{\leavevmode\raise.585ex\hbox{\small er}}
\theoremstyle{nonumberplain}
\theorembodyfont{\upshape}
\theoremseparator{:}
\newtheorem{definition}{Definition}

\theoremstyle{plain}
\theoremheaderfont{\normalfont\bfseries}
\theorembodyfont{\upshape}
\theoremseparator{:}

\theoremclass{remark}
\theoremstyle{break}

\theoremstyle{plain} \theoremheaderfont{\normalfont\bfseries}
\theorembodyfont{\slshape} \theoremseparator{.}
\newtheorem{theo}{Theorem}[section]
\theoremclass{theo}
\theoremstyle{break}

\theoremstyle{plain} \theoremheaderfont{\normalfont\bfseries}
\theorembodyfont{\slshape} \theoremseparator{.}

\newtheorem{cor}[theo]{Corollary}
\theoremstyle{plain} \theoremheaderfont{\normalfont\bfseries}
\theorembodyfont{\slshape} \theoremseparator{.}
\newtheorem{lemma}[theo]{Lemma}
\theoremstyle{plain} \theoremheaderfont{\normalfont\bfseries}
\theorembodyfont{\slshape} \theoremseparator{.}
\newtheorem{prop}[theo]{Proposition}
\newenvironment{proof}[1][Proof]{\textbf{#1.} }{\ \rule{0.5em}{0.5em}}

\DeclareMathOperator{\im}{im}

\DeclareMathOperator{\FIG}{FI[G]}
\DeclareMathOperator{\FIsG}{FI\#[G]}
\DeclareMathOperator{\C}{Conf}
\DeclareMathOperator{\Mod}{Mod}
\DeclareMathOperator{\PMod}{PMod}
\DeclareMathOperator{\PSigma}{P\Sigma}

\DeclareMathOperator{\Diff}{Diff}
\DeclareMathOperator{\PDiff}{PDiff}

\DeclareMathOperator{\Hom}{Hom}

\DeclareMathOperator{\Ind}{Ind}
\newcommand{\BPMod}{B\PMod}

\newcommand{\BDiff}{B\Diff}
\newcommand{\BPDiff}{B\PDiff}

\title{On the cohomology of pure mapping class groups as FI-modules } 
\author{ R. Jim\'enez Rolland} 
\date{}

\begin{document}

\maketitle
\begin{abstract} 

In this paper we apply the theory of finitely generated FI-modules developed by Church, Ellenberg and Farb to certain sequences of rational cohomology groups. Our main examples are the cohomology of the moduli space of $n$-pointed curves, the cohomology of the pure mapping class group of surfaces and some manifolds of higher dimension, and the cohomology of classifying spaces of some diffeomorphism groups. We introduce the notion of FI$[G]$-module and use it to strengthen and give new context to results on representation stability discussed by the author in a previous paper. Moreover, we prove that the Betti numbers of these spaces and groups are polynomial and find bounds on their degree. Finally, we obtain rational homological stability of certain wreath products.

\end{abstract}

\section{Introduction}
Our objective is to study some examples of sequences of cohomology groups that have an underlying structure of an FI-module as defined by Church--Ellenberg--Farb in \cite{3AMIGOS} and derive as many consequences as we can from this approach. 

We denote by {\bf FI} is the category whose objects
are finite sets and whose morphisms are injections.
First we will consider a functor $X$ from {\bf FI$^{op}$} to the category {\bf Top} of topological spaces or to the category {\bf Gp} of groups.  In the first case we call $X$ a {\it co-FI-space}, in the second case $X$ is a {\it co-FI-group}. Given such a functor $X$, for any  $i\geq 0$, we will be interested in the {\it FI-module} $H^i(X;R)$ over a Noetherian ring $R$. This is a functor from {\bf FI} to {\bf Mod$_{R}$}, the category of $R$-modules, that we obtain by composing $X$ with the cohomology functor $H^i(\_\_;R)$. We can consider the graded version:  $H^*(X;R)$ is called {\it graded FI-module over $R$}. In Section \ref{FI} below we recall the general definition of FI-module and some of its properties. 

In this paper we will focus in the FI-modules that arise from the  following examples of co-FI-spaces and co-FI-groups. 

\begin{itemize}
\item[{\bf (1)}]
The co-FI-space $\C_{\bullet}(M)$. For each $n\geq 1$ we consider the {\it configuration space $\C_n(M)$ of ordered $n$-tuples of distinct points on a topological space $M$}, which is 
the space of embeddings $\text{Emb}([n],M)$ of $[n]=\{1,\ldots,n\}$ into $M$. 
The co-FI-space $\C_{\bullet}(M)$ is given by {\bf n} $\mapsto \C_n(M)$ and for a given inclusion $f:[m]\hookrightarrow [n]$ in $\Hom_{\text{FI}}(\mathbf{m},\mathbf {n})$ the corresponding restriction $f^*:\C_n(M)\rightarrow \C_m(M)$ is given by precomposition.

\item[{\bf (2)}] The co-FI-space $\mathcal{M}_{g,\bullet}$. Let $g\geq 0$ and $n\geq 0$. We denote by $\mathcal{M}_{g,n}$ the {\it moduli space of genus $g$ Riemann surfaces with $n$ marked points}. The elements in $\mathcal{M}_{g,n}$ are equivalent classes, up to biholomorphism,  of pairs $(X, \mathfrak{p})$, where $X$ is a Riemann surface of genus $g$ and $\mathfrak{p}\in\C_n(X)$. The  functor $\mathcal{M}_{g,\bullet}$ is given by  $\mathbf{n}\mapsto\mathcal{M}_{g,n}$ and such that assigns to  $f\in\Hom_{\text{FI}}(\mathbf{m},\mathbf {n})$ the morphism $f^*: \mathcal{M}_{g,n}\rightarrow \mathcal{M}_{g,m}$  defined by $f^*\big([(X; \mathfrak{p})]\big)=[(X;\mathfrak{p}\circ f)]$.


\item[{\bf (3)}] The co-FI-group $\PMod^{\bullet}(M)$. Given $M$ a connected, smooth manifold, consider $\mathfrak{p}\in\C_n(\mathring{M})$, in other words, $\big(\mathfrak{p}(1),\mathfrak{p}(2),\ldots,\mathfrak{p}(n)\big)$ is  a configuration of $n$ distinct points in the interior of $M$. We denote by 
$\Diff^{\mathfrak{p}} (M)$   the subgroup in $\Diff (M \text{ rel } \partial M)$ of diffeomorphisms that leave invariant the set $\mathfrak{p}([n])$ and $\PDiff^{\mathfrak{p}} (M)$ is the subgroup that consists of the diffeomorphims that fix each  point in $\mathfrak{p}([n])$.

The \textit{mapping class group} is the group $\Mod^n(M):=\pi_0\big(\Diff^\mathfrak{p} (M)\big). $ Similarly the \textit{pure mapping class group} is $\PMod^n(M):= \pi_0(\PDiff^\mathfrak{p} (M)).$ 
Notice that if $\mathfrak{p},\mathfrak{q}\in\C_n(M)$, since $M$ is connected,  then $\Diff^{\mathfrak{p}} (M)\approx\Diff^{\mathfrak{q}} (M)$ and $\PDiff^{\mathfrak{p}} (M)\approx\PDiff^{\mathfrak{q}} (M)$. We refer to them by $\Diff^n(M)$ and $\PDiff^n(M)$, respectively.

The co-FI-group $\PMod^{\bullet}(M)$ is given by $\mathbf{n}\mapsto \PMod^{n}(M)$. Any inclusion
$f:[m]\hookrightarrow [n]$ induces a restriction map $\PDiff^{\mathfrak{p}}(M) \rightarrow \PDiff^{\mathfrak{p}\circ f}(M)$, which gives us the morphism $f^*:\PMod^n(M)\rightarrow\PMod^m(M)$.

When $M=\Sigma_{g,r}$ is a compact orientable surface of genus $g\geq 0$ with $r\geq 0$ boundary components, we restrict to orientation-preserving self-diffeomorphisms. We denote $\PMod^{\bullet}(\Sigma_{g,r})$ by $\PMod_{g,r}^{\bullet}$.

\item[{\bf (4)}] The co-FI-space $\BPDiff^\bullet(M)$.
This is the functor $\mathbf{n}\mapsto \BPDiff^n(M)$, where $\BPDiff^n(M)$ is the classifying space of the group $\PDiff^n$ defined before. The morphisms are defined in a similar manner as for $\PMod^{\bullet}(M)$. If $M$ is orientable, we can restrict to orientation-preserving diffeomorphisms.  

\end{itemize}

\noindent{\bf Remark: }Observe that for each of the previous examples, when $R=\mathbb{Q}$, the FI-module $H^i(X):=H^i(X;\mathbb{Q})$ encodes the information of the sequence $\big\{H^i(X_n;\mathbb{Q})\big\}_{n\in\mathbb{N}}$ of finite dimensional representations of the symmetric groups $S_n$. 

\subsection{Main Results}
Let $X$ be any of the co-FI-spaces or co-FI-groups from before.  A combination of the work by Church (\cite{CHURCH}), Church--Ellenberg--Farb (\cite{3AMIGOS}) and myself (\cite{JIM}) implies that, under some hypotheses recalled below, for any $i\geq0$ the FI-module $H^i(X)$ is {\it finitely generated} as defined in \cite{3AMIGOS} (see Section \ref{FI} for a precise definition).
In this paper we revisit these examples and give direct proofs of finite generation of $H^i(X)$ in Theorems \ref{DIM2},  \ref{MCG}, \ref{PUREDIM>3}, \ref{BDRY} and \ref{DIFFDIM>3} below. Then, the theory developed by Church--Ellenberg--Farb  in \cite{3AMIGOS} and further results by Church--Ellenberg--Farb--Nagpal in \cite{4AMIGOS} allow us to conclude the following:\medskip
\begin{theo}\label{EXA2} For $i\geq 0$, 
and each of the sequences $\{X_n\}$ in Table 1, there is an integer $N\geq 0$ such that for $n\geq N$ the following holds:
\begin{itemize}
\item[{\bf i)}] The decomposition of $H^i(X_n;\mathbb{Q})$ into irreducible $S_n$-representations stabilizes in the sense of uniform representation stability (as defined in \cite{CHURCH_FARB}) with stable range $n\geq N$.

\item[{\bf ii)}] 
The sequence of quotients $\{X_n/S_n\}$ satisfies rational homological stability for $n\geq N$. 

\item[{\bf iii)}] The character $\chi_n$  of $H^i(X_n;\mathbb{Q})$ is of the form:
$$\chi_{n}=Q_i(Z_1,Z_2,\ldots,Z_r),$$
where deg$(\chi_n)=r>0$ only depends on $i$ and $Q_i\in\mathbb{Q}[Z_1,Z_2,\ldots]$  is a unique polynomial in the class functions $$Z_l(\sigma):=\# \text{ cycles of length  l in }\sigma, \text{\hspace{5mm} for  any }\sigma\in S_n.$$

\item[{\bf iv)}] The length of the representation $\ell\big( H^i(X_n;\mathbb{Q})\big)$ is bounded above independently of $n$.\medskip 
\end{itemize}
\noindent The specific bounds for each example are presented in Table 1.
Moreover, if $k$ is any field, for each $X_n$ in Table 1 (except for $\mathcal{M}_{g,n}$) there exists an integer-valued polynomial $P(T) \in\mathbb{Q}[T]$ so that for all
sufficiently large $n$,
$$\dim_k\big(H^i(X_n;k)\big) = P(n).$$
For $X_n=\mathcal{M}_{g,n}$ this is true when $k=\mathbb{Q}$.
\end{theo}

\begin{table}[htbp]\label{TABLE}
\caption{Specific bounds}
$$
\begin{tabular}{|c|c|c|c|c|}
\hline
&  &  & &\\

{\bf $X_n$}& {\bf Hypotheses}& {\bf N}& {\bf $\ell\big(H^i(X_n;\mathbb{Q})\big)\leq$}&{\bf deg $(\chi_{n})\leq$} \\
& & & & \\
\hline
&  &  &  &  \\
 $\C_n(\Sigma)$&\begin{small}$\Sigma$ is a surface\end{small}&$5i$  &$2i+1$ &$2i$\\
\begin{footnotesize}\textsc{(Thm \ref{DIM2})}\end{footnotesize} 
&&  $4i$& & \\
&   &\begin{footnotesize}\text{(if $\partial \Sigma\neq\O$)}\end{footnotesize}  & & \\

&  &  &  & \\
 $\mathcal{M}_{g,n}$ &\begin{small}$g\geq 2$\end{small}  &$6i$  &$2i+1$ &$2i$ \\
\begin{footnotesize}\textsc{ (Thm \ref{MCG})}\end{footnotesize}&  &  & & \\
  &  &  & & \\

$\PMod^n_{g,r}$&\begin{small} $2g+r>2$, $r>0$\end{small} & $4i$  &$2i+1$ &$2i$\\
\begin{footnotesize}\textsc{ (Thm \ref{BDRY})}\end{footnotesize}  &  &  &  &  \\
&  &  & & \\

&\begin{small} $\dim M \geq 3$; $\pi_1(M)$ and $\Mod(M)$ are of  \end{small} & $3i$  &$i+1$ &$i$\\

$\PMod^n(M)$&\begin{small} type $FP_\infty$, and either $\pi_1(\Diff(M))=0$ \end{small} &$2i$ &&\\
\begin{footnotesize}\textsc{ (Thm \ref{PUREDIM>3})}\end{footnotesize} &\begin{small}or $\pi_1(M)$ has trivial center\end{small}  & \begin{footnotesize}\text{(if $\partial M\neq\O$)}\end{footnotesize}   & & \\
&  & & &\\

&\begin{small}$M$ is a smooth, compact and connected  \end{small}  & & & \\

 $\BPDiff^n(M)$&\begin{small} manifold; $\dim M\geq 3$, and $\BDiff(M \text{rel}\partial M)$  \end{small}   &$3i$  &$i+1$ &$i$\\
\begin{footnotesize}\textsc{ (Thm \ref{DIFFDIM>3})}\end{footnotesize} &\begin{small}has the homotopy type of a CW-complex with \end{small}     & & &\\
& \begin{small}  finitely many cells in each dimension\end{small}    &  & & \\
 &&  &  &  \\ 
\hline
\end{tabular}
$$
\end{table}


\noindent{\bf Remarks:}
\begin{itemize}

\item There is a vast literature on the cohomology of configuration spaces, with specific computations in special cases contained, for example, in \cite{COHEN_TAYLOR}, \cite{FELIX}, \cite{BOD_COHEN_TAYLOR}. Nevertheless, few explicit computations are known for the other examples in Table 1. Our main contribution is that we are able to get linear bounds  in $i$ for the degree of the polynomials $Q_i$ and the lengths of the representations. Moreover our new stable ranges for uniform representation stability are also linear in $i$, instead of the quadratic bounds in $i$ that were obtained in \cite{JIM}.

\item We believe that the representation stability results that we provide are compelling even in the cases where the homology is known. For instance, for the case of $\C_n(\mathbb{R}^2)$ where the cohomology algebra has been computed,  understanding multiplicities of specific irreducible representations is equivalent to various counts of irreducible squarefree polynomials with various properties (see \cite{3AMIGOS2}).  What representation stability says is that, in incredible generality, for many examples, these multiplicities stabilize in a very specific sense.

\item From Theorem \ref{EXA2} (iii) it follows that for each $\sigma\in S_n$, the character $\chi_{V_n}(\sigma)$ only depends on ``short cycles'', more precisely on the cycles of $\sigma$ of length $\leq $deg$(\chi_n)$.   

\item Our result for $\C_n(\Sigma)$ recovers the same stable range of $n\geq 4i$ obtained in \cite[Theorem 1]{CHURCH} for the case when $\Sigma$ is a closed surface or has non-empty boundary. The statment about $\dim_k\big(H^i(\C_n(\Sigma);k)\big)$ being a polynomial in $n$ for any field $k$ is a particular case of \cite[Theorem 1.8]{4AMIGOS}. \medskip

\item The moduli space $\mathcal{M}_{g,n}$  is a rational model for the classifying space $\BPMod_{g}^n$ for $g\geq 2$ (see for example \cite{HAINLOO} or \cite{HARERmoduli}). Hence
\begin{equation}\label{moduli}
H^*(\mathcal{M}_{g,n};\mathbb{Q})\approx H^*(\PMod_{g}^n;\mathbb{Q}).
\end{equation}

Therefore, for any $i\geq 0$ the FI-module $H^i\big(\mathcal{M}_{g,\bullet}\big)$ is precisely  $H^i\big(\PMod^\bullet_{g}\big)$  and Theorem \ref{MCG} below gives the corresponding bounds for this example. This theorem implies that the corresponding consistent sequences of rational $S_n$-representations $\{H^i(\PMod_{g,r}^n;\mathbb{Q})\}$ and  $\{H^i(\mathcal{M}_{g,n};\mathbb{Q})\}$ satisfy uniform representation stability with stable range $n\geq 4i$ when $r>0$  and $n\geq 6i$ in general. This improves the stable range $$n\geq \min \{4i+2(4g-6)(4g-5),2i^2+6i\}$$ obtained in \cite[Theorem 1.1]{JIM}.


\item Theorems \ref{PUREDIM>3} and \ref{DIFFDIM>3} below apply to irreducible, compact, orientable $3$-manifolds $M$ with nonempty boundary satisfying conditions (i)-(iv) in \cite[Section 3]{HATCHER_MC}.

\item It would be interesting to extend these results to local coefficient systems.  
\end{itemize}



\noindent{\bf A unified approach.}  In this paper we develop a unified approach to proving finite generation for FI-modules that arise as in the examples above.  In Section \ref{SS} we present a general spectral sequence argument that allows us to prove finite generation for our examples in Theorems \ref{DIM2},  \ref{MCG}, \ref{PUREDIM>3} and \ref{DIFFDIM>3}. Furthermore, this approach applies to spectral sequences arising from ``FI-fibrations'' over a fixed space and ``FI-group extensions'' of a given group (see Section \ref{FIBSPEC}). 

The basic idea is to use a spectral sequence of FI-modules converging to the graded FI-module of interest. We then use knowledge about finite generation of the FI-modules in the $E_2$-page and an inductive process together with closure properties of finite generation under subquotients and extensions to get our conclusion. The main difference between  Theorems \ref{DIM2} and the other theorems are the type of spectral sequence that we use and the way that finite generation is proved for the $E_2$-page. 

Given finite generation, the conclusions in Theorem \ref{EXA2} are consequences of \cite [Proposition 2.58  and Theorem 2.67]{3AMIGOS} and \cite[Theorem 1.2]{4AMIGOS}. 
\medskip

\noindent {\bf FI$[G]$-modules.} Let $G$ be a group. In Section \ref{FIGMod} we  introduce the notion of an {\it $\FIG$-module}: it is a functor $V$ from the category {\bf FI} to the category {\bf G-Mod} of $G$-modules over $R$. This definition incorporates the action of a group $G$ on our sequences of $S_n$-representations and allows us to take $V$ as twisted coefficients for cohomology. For $X$, a path-connected space with fundamental group $G$, and $p\geq 0$, we are interested in the FI-module $H^p(X;V)$ over $R$ given by $\mathbf{n}\mapsto H^p(X;V_n)$.
Our major result in Section \ref{FIGMod} is Theorem \ref{FIG} which  uses finite generation of an FI$[G]$-module $V$ to obtain finite generation and, when $R=\mathbb{Q}$,  specific bounds for the new FI-modules $H^p(X;V)$. This is our tool to prove  the base of the induction in the spectral sequence argument for Theorems \ref{MCG}, \ref{PUREDIM>3} and \ref{DIFFDIM>3}.\medskip

\noindent{\bf Remark:} It was pointed out to me by Ian Hambleton that FI$[G]$-modules can be understood in the framework of modules over EI-categories. 
An {\it EI-category $\Gamma$} is a small category in which each endomorphism is an isomorphism. An FI$[G]$-module corresponds to a left $R\Gamma$-module, where $R$ is the group ring $\mathbb{Z}G$ and $\Gamma$ is the EI-category $\textbf{FI}$. The theory of $R\Gamma$-modules  and its homological algebra have been developed and applied in the context of transformation groups (see for example \cite[Chapter I.11]{DIECK}).\medskip

\noindent{\bf Manifolds with boundary.} If we assume that $M$ is a manifold with non-empty boundary, the examples above of configuration spaces and pure mapping class groups have the extra structure of an FI\#-module that allows us to conclude the following results from the arguments in Section \ref{BDRYFIsharp}.

\begin{theo}\label{BETTI}
Let $\Sigma=\Sigma_{g,r}$ be a connected compact oriented surface with non-empty boundary ($r>0$).  For any $i\geq 0$ and $n\geq 0$, each of the following invariants of $\PMod^n_{g,r}$ is given by a polynomial in $n$ of degree at most $2i$:
\begin{itemize}
\item The $i^{th}$ rational Betti number $b_i(\PMod^n_{g,r})$ and the $i^{th}$ mod-$p$ Betti number of $\PMod^n_{g,r}$.
\item The rank of $H^i(\PMod^n_{g,r};\mathbb{Z})$ and the rank of the $p$-torsion part of $H^i(\PMod^n_{g,r}\mathbb{Z})$.
\end{itemize}
\end{theo}

\begin{theo}\label{BETTI2}
Let $M$ be a manifold with non-empty boundary that satisfies the hypothesis of Theorem \ref{PUREDIM>3}.  For $n\geq 0$ each of the following is given by a polynomial in $n$:
\begin{itemize}
\item The $i^{th}$ rational Betti number $b_i(\PMod^n(M))$ and the $i^{th}$ mod-$p$ Betti number of $\PMod^n(M)$. 
\item The rank of $H^i(\PMod^n(M);\mathbb{Z})$ and the rank of the $p$-torsion part of $H^i(\PMod^n(M);\mathbb{Z})$.
\end{itemize}
The polynomial is of degree at most $i$ for rational Betti numbers and degree at most $2i$ in the other cases.
\end{theo}


\noindent{\bf Closed Surfaces.} 
For a fixed $n \geq 0$, we can relate the mapping class group of a closed surface with the one of a surface with non-empty boundary.  Let $$\delta_g:\PMod^n_{g,1}\rightarrow\PMod^n_{g}$$
be the group homomorphism induced by gluing a disk to the boundary component. 
The following result is part of the so called Harer's stability Theorem and was proved initially by Harer (\cite{HARER}). A proof of it with the improved bounds that we use can be found in \cite{WAHL}.

\begin{theo} 
If $i\leq \frac{2}{3} g$,  we have following isomorphism:
$$H_i(\delta_g):H_i(\PMod^n_{g,1};\mathbb{Z})\rightarrow H_i(\PMod^n_{g};\mathbb{Z}).$$
\end{theo}

When the genus of the surface is large, by combining the previous result with Theorem \ref{BETTI} we obtain the following information in the case of closed surfaces. 

\begin{theo}\label{BETTI3}
If  $g\geq\max\{2,\frac{3}{2}i\}$, then each of the following invariants of $\PMod^n_{g}$ is given by a polynomial in $n$ for $n\geq 0$:
\begin{itemize}
\item the $i$-th rational Betti number $b_i(\PMod^n_{g})$
\item the rank of $H^i(\PMod^n_{g};\mathbb{Z})$
\item the rank of the $p$-torsion part of $H^i(\PMod^n_{g};\mathbb{Z})$
\end{itemize}
In each case the polynomial is of degree at most $2i$.
\end{theo}

\subsection{Cohomological stability of some wreath products}

{\bf Notation: } The {\it surface pure braid group} is the group $\pi_1\big(\C_n(\Sigma_{g,r})\big)$ and will be denoted by $P_n(\Sigma_{g,r})$. The {\it surface braid group} is  $\pi_1\big(\C_n(\Sigma_{g,r})/S_n\big)$ and we use $B_n(\Sigma_{g,r})$ to denote it. When $g=0$ and $r=1$, these are the pure braid group $P_n$ and the braid group $B_n$, respectively.  On the other hand, the {\it braid permutation group} $\Sigma_n^+$  is the group of string motions that preserve orientation of the circles (see \cite[Section 8]{WILSON} for a precise definition).\bigskip 

\noindent Let $\{K_n\}$ be a sequence of groups with surjections  $K_n\twoheadrightarrow S_n$. Given a group $G$ the wreath product $G\wr K_n$ is the semidirect product $G^n\rtimes K_n$, where $K_n$ acts on $G^n$ through the surjection $K_n\twoheadrightarrow S_n$. In Section \ref{WREATH} we discuss how our previous results and the closure of finite generation of  FI-modules under tensor products can be used to get information about homological stability of some wreath products. 

\medskip
 
\begin{theo}\label{COHOWREATH} Let $G$ be any group of type $FP_\infty$ and let $K_n$ be one of the following groups:  
\begin{itemize}
\item[(i)] The symmetric group $S_n$,
\item[(ii)] The surface braid group $B_n(\Sigma_{g,r})$, with $g,r\geq 0$,
\item[(iii)] The mapping class group $\Mod_{g,r}^n$, with $2g+r>2$,
\item[(iv)] The mapping class group $\Mod^n(M)$, where $M$ is a smooth connected manifold of dimension $d\geq 3$ such that the hypotheses in Theorem \ref{PUREDIM>3} are satisfied, 
\item[(v)] The braid permutation group $\Sigma_n^+$.
\end{itemize}



 
Then the wreath product $G\wr K_n$ satisfies rational homological stability.
\end{theo}

\noindent {\bf Remarks:}
In general we do not have explicit stable ranges. The following is known about stable ranges:
\begin{itemize}
\item For (i) we get the stable range $n\geq 2i$. Homological stability is known to hold integrally in this case for $n\geq 2i+1$ (see \cite[Propositions 1.6]{HATCHER_WAHL}). Therefore our bound suggests that the possible failure of injectivity when $n=2i$ should come from torsion.
\item For the case (ii), Hatcher--Wahl have shown that if $r>0$ the group  $G\wr B_n(\Sigma_{g,r})$ satisfies integral homological stability when $n\geq 2i+1$ (\cite[Propositions 1.7]{HATCHER_WAHL}). Rationally, the stable range has been improved to $n\geq 2i$ by Randall-Williams (see  \cite[Theorem A]{OSCAR}).
\end{itemize}

\subsection{Speculation on the existence of non-tautological classes in $\mathcal{M}_{g,n}$}
The {\it tautological ring} of $\mathcal{M}_{g,n}$ is defined to be a subring $\mathcal{RH}^*(\mathcal{M}_{g,n})$ of the cohomology ring $H^*\big(\mathcal{M}_{g,n};\mathbb{Q}\big)$ generated by certain ``geometric classes'': the kappa-classes $\kappa_j\in H^{2j}\big(\mathcal{M}_{g,n};\mathbb{Q}\big)$, for $j\geq 0$, and the psi-classes $\psi_i\in H^{2}\big(\mathcal{M}_{g,n};\mathbb{Q}\big)$, for $1\leq i\leq n$.  In $\mathcal{RH}^*(\mathcal{M}_{g,n})$, the class $\kappa_j$ has grading $j$ and $\psi_i$ has grading 1 (half the cohomological grading).  We refer the reader to \cite[Section 1]{FABERPAN} for precise definitions of the tautological rings of $\mathcal{M}_{g,n}$ and $\overline{\mathcal{M}}_{g,n}$.

In \cite[Section 5.1]{3AMIGOS} it is proved that $\mathcal{RH}^*(\mathcal{M}_{g,\bullet})$ is a graded FI-module of finite type for $g\geq 2$. This follows from the fact that this graded FI-module is a quotient of the free commutative algebra $$\mathbb{Q}\big[\{\kappa_j:j\geq 0\}\cup\{\psi_i:1\leq i\leq n\}\big],$$ where $S_n$ acts trivially on the kappa-classes  and permutes the psi-classes. From this description, we can see that for any $k\geq 0$ the weight of the FI-module  $\mathcal{RH}^k(\mathcal{M}_{g,\bullet})$ is at most $k$. As a consequence we obtain an upper bound for the length of the representation $\ell\big(\mathcal{RH}^k(\mathcal{M}_{g,n})\big)\leq k+1$.

On the other hand, Faber and Pandharipande studied in \cite{FABERPAN} the $S_n$-action on $H^*(\overline{\mathcal{M}}_{g,n};\mathbb{Q})$  and get an upper bound for the length of the irreducible representations occurring in the tautological ring $\mathcal{RH}^*(\overline{\mathcal{M}}_{g,n})$. Their interest is to exhibit, by other methods (counting, boundary geometry), several classes of Hodge type that cannot be tautological classes because the lengths of the corresponding $S_n$-representations are larger than their upper bound. In particular, they have established the existence of many non-tautological cohomology classes on $\overline{\mathcal{M}}_{2,21}$.  They obtained that $\ell\big(\mathcal{RH}^k(\overline{\mathcal{M}}_{g,n})\big)\leq k+1$ (\cite[Section 4]{FABERPAN}). Since $\mathcal{RH}^k(\overline{\mathcal{M}}_{g,n})$ surjects onto $\mathcal{RH}^k(\mathcal{M}_{g,n})$, that implies that $\ell\big(\mathcal{R}^k(\mathcal{M}_{g,n})\big)\leq k+1$, which is the same bound that we obtained directly with the FI-module approach. In contrast, their method involves studying representations induced from the boundary strata.

Finally we would like to point out that from Table 1 we have the upper bounds  $$\ell\big(H^{2k}(\mathcal{M}_{g,n};\mathbb{Q})\big)\leq 4k+1.$$ This, contrasted with $\ell\big(\mathcal{RH}^k(\mathcal{M}_{g,n})\big)\leq k+1$, suggests that there is room for the existence of non-tautological classes $\mathcal{M}_{g,n}$ and that an approach \`a la  Faber and Pandharipande could demonstrate that some explicit classes are non-tautological. However, we have no indication that our bounds are sharp. As matter of fact,  the only completely known case shows evidence of the contrary since $H^{2}\big(\mathcal{M}_{g,n};\mathbb{Q}\big) =\mathcal{RH}^1(\mathcal{M}_{g,n})$ 
 has length $2$. \bigskip

\noindent{\bf Acknowledgements.} I am grateful to Tom Church, Benson Farb, Peter May and Jenny Wilson for many useful discussions and comments on earlier drafts of this paper. I also want to thank Jordan Ellenberg for his recommendation of the notation for FI$[G]$-modules and Ian Hambleton for helpful comments. I am indebted to the anonymous referees for their corrections and suggestions.

\section{Preliminaries}\label{PRE}

In this section we sumarize notions introduced in \cite{3AMIGOS} and \cite{CHURCH_FARB} and set notation that will be recurrent in the rest of the paper. We refer the interested reader to \cite[Sections 1 \& 2]{3AMIGOS} for precise results and proofs.

\subsection{FI and FI\#-modules}\label{FI}

Let {\bf FI} be the category whose objects are natural numbers {\bf n} and whose morphisms {\bf m}$\rightarrow${\bf n} are injections from $[m]:=\{1,\ldots, m\}$ to $[n]:=\{1,\ldots, n\}$.  Similarly we denote by {\bf FI\#}  the category whose objects are natural numbers {\bf n} and the morphisms {\bf m }$ \mapsto ${ \bf n} are  triples $(A,B,\psi)$, where $A\subset [m]$, $B\subset [n]$ and $\psi: A\rightarrow B$ is a bijection.

\begin{definition} 
An {\it FI-module} over a commutative ring $R$ is a functor $V$ from the category {\bf FI} to the category  {\bf Mod}$_R$ of modules over $R$.   An  {\it FI\#-module} over $R$ is a functor $V$ from {\bf FI\#} to {\bf Mod}$_R$. We denote $V(${\bf n}) by $V_n$ and $V(f)$ by $f_*$, 
  for any $f\in\text{Hom}_{\text {FI}}(\mathbf{m},\mathbf{n})$. In the same manner a functor $V=\bigoplus V^i$ from {\bf FI} to the category of graded modules over $R$ is called a {\it graded FI-module} over $R$. In particular, each $V^i$ is an FI-module over $R$.
\end{definition}

The category {\bf FI-Mod}$_R$ of FI-modules over $R$ is an abelian category. The concepts of kernel, cokernel, sub-FI-module, quotient, injection and surjection are defined ``pointwise'' . 
Most of the examples that we consider are finite dimensional vector spaces over $\mathbb{Q}$, unless otherwise specified. Therefore, we use the notation {\bf FI-Mod} for the category of FI-modules over $\mathbb{Q}$.

In the category {\bf FI-Mod}$_R$ we can define analogous concepts of the basic definitions coming from module theory. 

\begin{definition}
An FI-module $V$ over $R$  is said to be {\it finitely generated in degree $\leq m$} if there exist $v_1,\ldots,v_s$,  with each $v_i\in V_{n_i}$ and $n_i\leq m$, such that $V$ is the minimal sub-FI-module of $V$ containing $v_1,\ldots,v_s$. We write $V=\text{span}(v_1,\ldots,v_s)$. An FI\#-module  over $R$ is {\it finitely generated in degree $\leq m$} if the underlying FI-module is finitely generated in degree $\leq m$. A graded FI-module  $V$ over $R$ is said to be of {\it finite type} if each FI-module $V^i$ is finitely generated.
\end{definition}

Finitely generated FI-modules have strong closure properties that allow our arguments below. In particular, 
extensions and quotients of finitely generated FI-modules are still finitely generated (\cite[Proposition 2.17]{3AMIGOS}). 
Furthermore they satisfy a ``Noetherian property'' in the following sense: If $V$ is a finitely generated FI-module over a Noetherian ring $R$, and $W$ is a sub-FI-module of $V$ , then $W$ is finitely generated. This is \cite[Theorem 1.1]{4AMIGOS}. It was first proved for FI-modules over a field of characteristic zero in \cite[Theorem 2.60]{3AMIGOS}.

In some of our examples below,  the FI-module $V$ over $R$ has actually the extra structure of an FI\#-module over $R$. In that case more is true: if $V$ is finitely generated in degree $\leq m$, then any sub-FI\#-module is finitely generated in degree $\leq m$ (follows from \cite[Corollaries 2.25 \& 2.26]{3AMIGOS}). \medskip

\noindent {\bf Notation (The FI-modules $M(W)$): } Let $m\in\mathbb{N}$ and consider a fixed $S_m$-representation $W$ over a field $k$ or $k=\mathbb{Z}$. The FI-module $M(W)$ is defined as follows:

\begin{displaymath}
   M(W)_n:= \left\{
     \begin{array}{lr}
      0, &  \text{if } n<m\\
       \Ind_{S_{m}\times S_{n-m}}^{S_n}W\boxtimes k, & \text{if } n\geq m.
     \end{array}
   \right.
\end{displaymath} 

In particular, when $W=k[S_m]$ we will denote the FI-module $M(W)$ by $M(m)$. These FI-modules were introduced in \cite[Section 2.1]{3AMIGOS}. By definition, they are finitely generated in degree $m$. Moreover they have the structure of an FI\#-module and have surjectivity degree at most $m$. 
The FI\#-modules $M(W)$ are ``building blocks'' for general FI\#-modules (\cite[Theorem 2.24 and Corollary 2.26]{3AMIGOS}).

\subsection{FI-modules over fields of characteristic zero}

\noindent {\bf Notation (Representations of $S_n$ in characteristic zero): }
The irreducible representations of $S_n$ over a field of characteristic zero  $k$ are classified by partitions $\lambda$ of $n$. By a partition of $n$ we mean $\lambda=(\lambda_1\geq\cdots\geq\lambda_l >0)$ where $l\in\mathbb{Z}$ and $\lambda_1+\cdots+\lambda_l=n$.  We will write $|\lambda|=n$. The corresponding irreducible $S_n$-representation will be denoted by $V_{\lambda}$. Every  $V_{\lambda}$ is defined over $\mathbb{Q}$ and any $S_n$-representation decomposes over $\mathbb{Q}$ into a direct sum of irreducibles (\cite{FULTON_HARRIS} is a standard reference). The decomposition of an $S_n$-representation over any such field $k$ does not depend on $k$.

If $\lambda$ is any  partition of $m$, i.e. $|\lambda|=m$, then for any $n\geq |\lambda|+ \lambda_1$ the \textit{padded partition} $\lambda[n]$ of $n$ is given by $\lambda[n]=(n-|\lambda|,\lambda_1,\cdots,\lambda_l)$. Keeping the notation from \cite{CHURCH_FARB}  we set $V(\lambda)_n=V_{\lambda[n]}$  for any $n\geq |\lambda|+\lambda_1$. Every irreducible $S_n$-representation is of the form $V(\lambda)_n$ for a unique partition $\lambda$.  
For a given partition $|\lambda|=m$, we use $M(\lambda)$ to denote the FI-module $M(V_{\lambda})$.
We define the {\it length} of an irreducible representation of $S_n$ to be the
number of parts in the corresponding partition of $n$. The trivial representation has length $1$, and the alternating representation has length $n$. We define the {\it length $\ell(V )$} of a finite dimensional representation $V$
of $S_n$ to be the maximum of the lengths of the irreducible constituents. Notice that $\ell(V_\lambda)\leq|\lambda|$.
\bigskip

The proof of the ``Noetherian property'' does not give an upper bound on the degree in which a given subobject is generated. To deal with this for FI-modules over fields of characteristic zero the notion of {\it weight of an FI-module} was introduced in \cite[Section 2.5]{3AMIGOS}.

In this subsection let $k$ be a field of characteristic zero.

\begin{definition} Let $V$ be an FI-module over $k$.
We say that $V$ has {\it weight $\leq d$}  if for every $n\geq 0$ and every irreducible constituent $V(\lambda)_n$ we have $|\lambda|\leq d$.
\end{definition}

Notice that the weight of an FI-module is closed under subquotients and extensions.


The subgroup of $S_n$ that permutes $\{a+1,\ldots,n\}$ and acts trivially on $\{1,2,\ldots,a\}$ is denoted by $S_{n-a}$. The coinvariant quotient $(V_n)_{S_{n-a}}$ is the $S_a$-module $V_n\otimes_{k[S_{n-a}]}k$, i.e. the largest quotient of $V_n$ on which $S_{n-a}$ acts trivially.

The following provides a notion of stabilization and range of stabilization for an FI-module (this is just a rephrasing of \cite[Definitions 2.34 \& 2.35]{3AMIGOS}).

\begin{definition}
Let $V$ be an FI-module over $k$. If for every $a\geq 0$ and $n\geq N+a$ the map of coinvariants
\begin{equation}\label{MAPI}
(V_n)_{S_{n-a}}\rightarrow (V_{n+1})_{S_{n-a}}
\end{equation}
induced by the standard inclusion $I_n:\{1,\ldots n\}\hookrightarrow\{1,\ldots, n, n+1\}$, is an  injection of $S_a$-modules,  we say that $V$ has {\it injectivity  degree $\leq N$}.  If the map (\ref{MAPI}) is surjective, we say that $V$ has {\it surjectivity degree} $\leq N$.  The FI-module $V$ has {\it stability type} $(M,N)$ if it has injectivity degree $M$ and surjectivity degree $N$. When $V$ is an $FI\#$-module, the identity on $V_n$ factors through $I_n$, hence the injectivity degree is always $0$.

If the FI-module $V$ has stability type $(M,N)$, then we say that the {\it stability degree} of $V$  is given by at most $\max(M,N)$. 
\end{definition}

\section{A spectral sequence argument}\label{SS}
In this section we present the general spectral sequence argument that will give us finite generation and specific bounds in Theorems \ref{DIM2} and \ref{LERAYSERRE}. We basically apply the idea used in the proof of \cite[Theorem 4.2]{3AMIGOS} to a more general context. \bigskip

\noindent {\bf Setting:} Suppose that we have a first quadrant spectral sequence of FI-modules $E^{p,q}_*$ over $\mathbb{Q}$ converging to a graded FI-module $H^*(E)$ over $\mathbb{Q}$. 
Let $\alpha$  and $\beta$ be two non-negative constants such that $2\alpha\leq \beta$. In what follows, we assume that for any $p,q\geq0$ the FI-module  $E_2^{p,q}$ is finitely generated with  injectivity degree at most $\beta q$ and surjectivity degree at most $\alpha p+\beta q$.

In our applications below, $E^{p,q}_*$ is either a Leray, Leray--Serre or Hochschild--Serre spectral sequence.
\begin{lemma}\label{Er}  For any $p,q\geq0$ and $r\geq 3$,  the FI-module $ E_r^{p,q}$ is finitely generated with injectivity degree at most $\alpha p+\beta q+(\beta-\alpha)r+(\alpha-2\beta)$ and surjectivity degree at most $\alpha p+\beta q$.
\end{lemma}
\begin{proof} Finite generation  of an FI-module is closed under subquotients. To verify the stated stability type we proceed by induction on $r\geq 3$. The base of induction is the case $r=3$. To compute $E_{3}^{p,q}$ we consider the complex of FI-modules
 \begin{equation*}
\xymatrix{
E_2^{p-2,q+1} \ar[r]& E_2^{p,q} \ar[r]& E_2^{p+2,q-1}},
\end{equation*}
 where the left map is the differential $d_2^{p-2, q+1}$ and the right map is $d_2^{p,q}$.
By hypothesis the left hand side term in the previous complex has surjectivity degree at most $\alpha(p-2)+\beta(q+1)=\alpha p+\beta q +(\beta-2\alpha)$. The middle term has stability type at most  $\big(\beta q ,\alpha p+\beta q\big)$  and  the right hand side term has injectivity degree at most $\beta(q-1)$. Hence, by applying \cite[Proposition 2.45]{3AMIGOS} to the complex of FI-modules above, we obtain that the quotient FI-module $$E_{3}^{p,q}\approx \ker d_2^{p,q}/\im d_2^{p-2, q+1}$$ has injectivity degree at most 
$$\max\big(\alpha p+\beta q +(\beta-2\alpha),\beta q\big)= \alpha p+\beta q +(\beta-2\alpha)=\alpha p+\beta q+(\beta-\alpha)(3)+(\alpha-2\beta)$$ since $2\alpha\leq\beta$, and surjectivity degree at most $$\max\big(\alpha p+\beta q,\beta q-\beta\big)=\alpha p+\beta q$$ since $\alpha,\beta\geq 0$. 

Now suppose that the statement is true for $E_r^{p,q}$. To compute $E_{r+1}^{p,q}$ we consider the complex of FI-modules
 \begin{equation*}
\xymatrix{
E_r^{p-r,q+r-1} \ar[r]& E_r^{p,q} \ar[r]& E_r^{p+r,q-r+1}},
\end{equation*}
 where the left map is the differential $d_r^{p-r, q+r-1}$ and the right map $d_r^{p,q}$.
By induction, the left hand side term in the previous complex has surjectivity degree at most $\alpha p+\beta q+(\beta-\alpha)(r+1)+(\alpha-2\beta)$. The middle term has stability type at most  $$\big(\alpha p+\beta q+(\beta-\alpha)r+(\alpha-2\beta),\alpha p+\beta q\big).$$ Finally the right hand side term has injectivity degree at most $\alpha p+\beta q+\alpha-\beta$. By applying again \cite[Proposition 2.45]{3AMIGOS} we get the desired stability type for the quotient $$E_{r+1}^{p,q}\approx \ker d_r^{p,q}/\im d_r^{p-r, q+r-1}.$$
\end{proof}

For a given $i\geq 0$ and $0\leq p\leq i$, we have that  $E_{\infty}^{p,i-p}=E_{i+2}^{p,i-p}$. From Lemma \ref{Er} we get the immediate corollary. 
\begin{cor}\label{Einfinity} 
The FI-module $E_{\infty}^{p,i-p}$
has injectivity degree  at most $$\alpha p+\beta (i-p)+(\beta-\alpha)(i+2)+(\alpha-2\beta)=(2\beta-\alpha) i +(\alpha-\beta) p-\alpha\leq (2\beta-\alpha) i-\alpha$$ and surjectivity degree at most $$\alpha p+\beta (i-p)=\beta i+(\alpha-\beta) p\leq \beta i.$$  
\end{cor}

As assumed at the beginning of this section, the spectral sequence $E_*^{p,q}$ converges to a graded FI-module $H^*(E)$. From Lemma \ref{Er} and Corollary \ref{Einfinity}  we can conclude the following about the stability type of each FI-module $H^i(E)$.
\begin{theo}\label{MAINSPEC} 
Suppose that we have a first quadrant spectral sequence of FI-modules $E^{p,q}_*$ over a Noetherian ring $R$ converging to a graded FI-module $H^*(E;R)$ over $R$. If the $FI$-module  $E_2^{p,q}$ is finitely generated, then the FI-module $H^i(E;R)$ is  finitely generated, for any $i\geq 0$.

Furthermore, asume that $R=\mathbb{Q}$ and that for any $p,q\geq0$ the FI-module  $E_2^{p,q}$ has injectivity degree at most $\beta q$ and surjectivity degree at most $\alpha p+\beta q$ , where $\alpha, \beta\geq 0$ such that $2\alpha\leq \beta$. Then, the FI-module $H^i(E)=H^i(E;\mathbb{Q})$ is  finitely generated with  stability type at most $((2\beta-\alpha)i-\alpha, \beta i)$.
\end{theo}
\begin{proof}
The first statement follows from the fact that finite generation of an FI-module over a Noetherian ring $R$ is closed under subquotients (\cite[Theorem 1.1]{4AMIGOS}). 

For each $i\geq 0$, there is a natural filtration of $H^i(E)$ by FI-modules  
\begin{equation}\label{FILTRATION}
0\subseteq F^i_i\subseteq F^i_{i-1}\subseteq \ldots\subseteq F^i_1\subseteq F^i_0=H^i(E),\end{equation}  where, for $0\leq p\leq i$, the successive quotients $F^i_p/F^i_{p+1}\approx E_{\infty}^{p,i-p}$. The second statement for the case $k=\mathbb{Q}$ follows from combining the bounds in Lemma \ref{Einfinity} with \cite[Proposition 2.46]{3AMIGOS}, which states injectivity and surjectivity degrees for filtrations of FI-modules satisfying the conditions above.  
\end{proof}\bigskip


\subsection{Spectral sequences and FI\#-modules}\label{SSFIsharp}
We conclude this section with an argument that allows us to take advantage of the extra structure of finitely generated FI\#-modules to get information about the cases where $k$ is a field of arbitrary characteristic or $\mathbb{Z}$. This follows essentially the proof of \cite[Theorem 4.7]{3AMIGOS}.\bigskip

\noindent {\bf Setting:} Suppose that we have a first quadrant spectral sequence of FI-modules $E^{p,q}_*$ over $k$ converging to a graded FI\#-module $H^*(E;k)$ over $k$.
Let $\alpha$  and $\beta$ be two non-negative constants such that $\alpha\leq \beta$. Assume that for any $p,q\geq0$ each term $E_2^{p,q}$ is an FI\#-module which is finitely generated in degree $\leq \alpha p+\beta q$.

\begin{theo}\label{SPECFIsharp} Let $k$ be any field or $\mathbb{Z}$. For any $i\geq 0$ the FI\#-module $H^i(E;k)$ is  finitely generated in degree $\leq\beta i$.\end{theo}
\begin{proof}
Suppose first that $k$ is a field. We have that $E_2^{p,q}$ is an FI\#-module which is finitely generated in degree $\leq \alpha p+\beta q$.  \cite[Corollary 2.27]{3AMIGOS} allows to relate this upper bound on the degree of generation with the dimension of the $k$-vector space $E_2^{p,q}(n)$ and conclude that $\dim_k E_2^{p,q}(n)=O(n^{\alpha p+\beta q})$. Since $E_\infty^{p,q}$ is a subquotient of $E_2^{p,q}$ and $k$ is a field, then $\dim_k E_\infty^{p,q}(n)=O(n^{\alpha p+\beta q})$. Finally for each $i\geq 0$, from the filtration (\ref{FILTRATION}) of $H^i(E;k)$, we have that $\dim_k H^i(E;k)=O(n^{\beta i})$ (since $\alpha p+\beta (i-p)\leq \beta i$ for any $0\leq p\leq i$). Hence, by applying again \cite[Corollary 2.27]{3AMIGOS} we get the desired implication.
The case when $k=\mathbb{Z}$ can be treated similarly because the rank of a $\mathbb{Z}$-module is non-increasing when passing to submodules.
\end{proof}\bigskip

\section{Sequences of cohomology groups as FI-modules (part I)}\label{EG1}

In this section we revisit two examples of FI-modules that are key ingredients to understand our main examples in Section \ref{EG2}.

\subsection{The  FI-module $H^i(M^{\bullet};k)$}

Given $M$ a topological space, consider the co-FI-space $M^\bullet$. 
It is the functor that assigns $\mathbf{n}\mapsto M^n:=\underbrace{M\times\cdots\times M}_{n}$. Morphisms are defined as follows: 

If $f\in\Hom_{\text{FI}}(\mathbf{m},\mathbf {n})$, then   $f^*:M^n\rightarrow M^m$ is given by $f^*(x_1,\ldots,x_n)=(x_{f(1)},\ldots, x_{f(m)})$.  For each $i\geq 0$ we compose with the contravariant functor $H^i(\_\_;k)$ to get an FI-module over a field $k$.

\begin{prop}\label{PROD} Let $k$ be a field and $M$ be a connected CW-complex with   $\dim_{k}\big(H^i(M;k)\big)<\infty$ for any $i\geq 0$. Then   $H^i(M^{\bullet};k)$ is an FI\#-module finitely generated over $k$. If $k=\mathbb{Q}$ then it has weight $\leq i$ and has stability type at most $(0,i)$.
\end{prop}
\begin{proof}
This is a consequence of the K\"unneth formula. As pointed out in \cite[Section 4]{3AMIGOS} the graded FI-module $H^*(M^\bullet;k)$ coincides, apart from signs, with the graded FI-module $H^*(M;k)^{\otimes\bullet}$  (see \cite[Definition 2.71]{3AMIGOS}) which is a finitely generated FI$\#$-module since $M$ is connected and $\dim_{k}\big(H^i(M;k)\big)<\infty$. When $k=\mathbb{Q}$, it can actually be shown that $H^i(M^{\bullet};\mathbb{Q})$ is a direct sum of FI$\#$-modules of the form $M(W_j)$, where $W_j$ is some $S_j$-representation and each summand satisfies that $j\leq i$ (see for example \cite[Proposition 6.5]{JIM}). Then the weight and the stability type claimed in Proposition \ref{PROD} follow.
\end{proof}\medskip

\noindent 

Similarly, if $G$ is a group, we can consider the co-FI-group $G^{\bullet}$.
If $G$ is a group of type $FP_{\infty}$ (see for example \cite[Chapter VIII]{BROWN} for definition), then the CW-complex $M=K(G,1)$  satisfies the hypotheses in Proposition \ref{PROD} and it follows that the FI$\#$-module $H^i(G^{\bullet};k)=H^i(M^{\bullet};k)$ is finitely generated .
 
\subsection{Cohomology of configuration spaces}

Let $k$ be any field and let $M$ be a connected, oriented manifold of dimension $d\geq 2$ and assume that $\dim_{k}\big(H^*(M);k)\big)<\infty$. 
Since the inclusion $\C_n(M)\hookrightarrow M^n$ is $S_n$-equivariant, we get a corresponding map of co-FI-spaces $\C_{\bullet}(M)\rightarrow M^{\bullet}.$  We recall here how a spectral sequence argument can be used to obtain finiteness conditions for the FI-modules  $H^q\big(\C_{\bullet}(M);k\big)$.

Let us take, together for all $n$, the Leray spectral sequences of $\C_n(M)\hookrightarrow M^n$. The functoriality of the Leray spectral sequence implies that we have a spectral sequence of FI-modules $$E^{p,q}_*=E^{p,q}_*(\C_{\bullet}(M)\rightarrow M^{\bullet})$$ converging to the graded FI-module $H^*(\C_{\bullet}(M);k)$. 

Using this spectral sequence, finite type of this graded FI-module has been proved over any field $k$ in \cite[Proposition 4.1]{4AMIGOS}. For the case when $k=\mathbb{Q}$ and the dimension of M is $d\geq 3$, particular bounds for the stability degree have been obtained.

\begin{theo}[\cite{3AMIGOS}, Theorem 4.2 ]\label{DIM3}
Let $M$ be a connected, oriented manifold of dimension $d\geq 3$. For any $i\geq 0$, the FI-module $H^i(\C_{\bullet}(M);\mathbb{Q})$ has weight $\leq i$ and stability type at most $(i+2-d, i)$. 
\end{theo} 

\noindent We now focus in the case where $\Sigma$ is a connected, oriented surface ($d=2$).
Following the approach in \cite[Section 4]{3AMIGOS} we get a better bound for the degree of the FI-module $H^i\big(\C_{\bullet}(\Sigma);\mathbb{Q}\big)$ and get the specific bounds for the stability type.

\begin{theo} \label{DIM2}Let $\Sigma$ be a connected, oriented manifold of dimension $2$. For any $i\geq 0$, the FI-module $H^i(\C_{\bullet}(\Sigma);\mathbb{Q})$ is finitely generated of weight $\leq 2i$ and has stability type at most $(2i,2i)$ when $\Sigma$ is a closed surface, at most $(0,2i)$ when $\partial \Sigma$ is nonempty, and at most $(3i-1,2i)$ otherwise.
\end{theo}

\noindent\begin{proof}
We have a spectral sequence of FI-modules $$E^{p,q}_*=E^{p,q}_*(\C_{\bullet}(\Sigma)\rightarrow \Sigma^{\bullet})$$ converging to the graded FI-module $H^*(\C_{\bullet}(\Sigma))$. 
For any $p,q\geq 0$ the FI-module $E_2^{p,q}$ is the direct sum of FI-modules of the form $M(W_k)$ where $W_k$ is a certain $S_k$-representation. Moreover, each summand satisfies $k\leq p+2q$ (see \cite[Section 3.3.]{CHURCH}). Hence, for every $p,q\geq 0$  we have that $E_2^{p,q}$ is finitely generated in degree $\leq p+2q$ and has stability type at most $(0,p+2q)$. This is precisely the setting needed for our spectral sequence argument in Section \ref{SS} with constants $\alpha=1$ and $\beta=2$. 
Then for each $i\geq 0$ the FI-module $H^i(\C_{\bullet}(\Sigma))$ is finitely generated with stability type at most $(3i-1,2i)$.

In addition, Totaro proved in \cite[Theorem 3]{TOTARO} that if $M$ is a smooth complex projective variety, then  $E_\infty\big(\C_n(M)\hookrightarrow M^n\big)=E_3$. This is the case when $\Sigma$ is a closed surface. Therefore we can use Lemma \ref{Er} to improve the bounds for the stability type of $E_\infty^{p,i-p}$ to be at most $(2i,2i)$, which gives the corresponding bounds stated before. 

On the other hand, if $\partial \Sigma$ is nonempty, \cite[Proposition 4.6]{3AMIGOS} implies that  $H^i(\C_\bullet(\Sigma))$ has an $FI\#$-module structure and the injectivity degree is $0$.

Finally, observe that for any $0\leq i$ and $0\leq p\leq i$ the FI-module $E_{\infty}^{p,i-p}$ is a subquotient of  the FI-module $E_2^{p,i-p}$ of weight $p+2q$. It follows that weight$(E_{\infty}^{p,i-p})\leq 2i-p\leq 2i$, which implies that $H^i(\C_{\bullet}(\Sigma))$ has weight at most $2i$.
\end{proof}\bigskip

\noindent {\bf Remark:}  In \cite[Section 2.6]{3AMIGOS} finite generation of an FI-module is related with representation stability. In particular, Theorem \ref{DIM2} together with \cite[Proposition 2.58]{3AMIGOS} imply that the sequence $H^i(\C_n(\Sigma);\mathbb{Q})$ is uniformly representation stable and we recover the stable range of $n\geq 4i$ for the cases when $\Sigma$ is closed or has non-empty boundary, which was first obtained in \cite[Theorem 1]{CHURCH}.\bigskip

\noindent For manifolds with non-empty boundary it follows from \cite [Proposition 4.6]{3AMIGOS} that $H^i(\C_\bullet(M);R)$ is an FI\#-module for any commutative ring $R$. The argument of Section \ref{SSFIsharp} implies the following result.

\begin{theo}[Theorem 4.7 in \cite{3AMIGOS}]\label{CONF_BDRY}
Let $M$ be a connected, oriented manifold of dimension $d\geq 2$ which is the interior of a compact manifold with non-empty boundary. Let $k$ be any field or $\mathbb{Z}$, then for each $i\geq 0$ the FI\#-module $H^i(\C_\bullet(M);k)$ over $k$ is finitely generated by $O(n^{2i})$ elements.  
\end{theo} 

\section{FI$[G]$-modules}\label{FIGMod}
Here we introduce the notion on an FI$[G]$-module. Basically we want to incorporate the action of a group $G$ on our sequences of $S_n$-representations. These types of FI-modules will allow us to construct new FI-modules by taking cohomology with twisted coefficients. We will see how in some situations we can use finite generation of the original FI$[G]$-module to get finite generation and specific bounds for the new FI-modules. In Section \ref{FIBSPEC} we use this setting in spectral sequence arguments for cohomology of fibrations and groups extensions.

\begin{definition} Let $R$ be any commutative ring and let $G$ be a group. An  {\it FI$[G]$-module} $V$ over $R$ is a functor from the category {\bf FI} to the category {\bf G-Mod$_R$} of $G$-modules over $R$.  We say that an FI$[G]$-module $V$ is {\it finitely generated} if it is finitely generated as an FI-module. Similarly an {\it $\FIsG$-module} $V$ over $R$ is a functor from the category {\bf FI\#} to the category {\bf G-Mod$_R$}. \end{definition}

\noindent {\bf FI$[G]$-modules and consistent sequences compatibles with $G$-actions:}
For an FI$[G]$-module $V$, for each $\sigma\in S_n$ the  induced linear automorphim  $\sigma_*:V_n\rightarrow V_n$ is a $G$-map. Hence the $S_n$-action and the $G$-action on $V_n$ commute. If we denote by $\phi_n$ the map obtained by applying $V$ the standard inclusion $I_n$ (i.e. $\phi_n=V(I_n)$)
, we have that  $\{V_n, \phi_n\}$ is a {\it consistent sequence of $S_n$-representations compatible with $G$-actions} in the sense of \cite{JIM}.

\subsection{Getting new FI-modules from  FI$[G]$-modules }
 Let $V$ be an FI$[G]$-module over $R$. Consider a path connected space $X$ with fundamental group $G$. For each integer $p\geq 0$ we have a covariant functor $H^p(X;\_\_)$  from the category {\bf G-Mod$_R$} to the category {\bf Mod}$_R$. Hence we have a new  FI-module $H^p(X;V)$ over $R$ where $H^p(X;V)_n:=H^p(X;V_n)$, the $p$th cohomology of $X$  with local coefficients in the $G$-module $V_n$ (see \cite[Section 3.H]{HATCHER}). Moreover the functor $H^*(X;V)$ given by  $H(X;V)_n :=H^*(X;V_n)$  is a graded FI-module over $R$.

\begin{theo}[Cohomology with coefficients in a f.g. $\FIG$-module]\label{FIG} 

Let $G$ be the fundamental group of a connected CW complex $X$ with finitely many cells in each dimension. If $V$ is a finitely generated FI$[G]$-module over a Noetherian ring $R$, then for every $p\geq 0$, the FI-module $H^p(X;V)$ is finitely generated over $R$.

Moreover, if $R=\mathbb{Q}$ and $V$ has weight $\leq m$ and stability degree $N$, then the FI-module $H^p(X;V)$ has weight $\leq m$ and stability degree $N$.
\end{theo}

\noindent\begin{proof}
Given that $G=\pi_1(X)$, the universal cover $\tilde{X}$ of $X$  has a $G$-equivariant cellular chain complex. Since $X$ has finitely many cells in each dimension, for each $p\geq 0$ the group $C_p(\tilde{X})$  is a free $G$-module of finite rank, say $C_p(\tilde{X})\approx (\mathbb{Z}G)^{d_p}$. A preferred $G$-basis $x_1,\ldots,x_{d_p}$ can be provided by selecting a $p$-cell in $\tilde{X}$ over each cell in $X$.

For each $p\geq 0$ and $n\in\mathbb{N}$ we have an isomorphism of $G$-modules   $\mathcal{H}om_G(C_p(\tilde{X}),V_n)\approx V_n^{\oplus d_p}$, given by $h\mapsto\big(h(x_1),\ldots,h(x_{d_p})\big)$.
Moreover, for any morphism $\phi: V_m\rightarrow V_n$, the following diagram commutes:
\begin{equation*}
\xymatrix{
\mathcal{H}om_G(C_p(\tilde{X}),V_m)\ar[d]_{\approx}\ar[r]^{\phi\circ\_}&\mathcal{H}om_G(C_p(\tilde{X}),V_{n})\ar[d]^{\approx}\\
V_m^{\oplus d_p}\ar[r]^{\phi^{\oplus d_p}}&V_n^{\oplus d_p}}
\end{equation*}

Therefore the FI$[G]$-module $C^p(X;V)$ given by $C^p(X;V)_n:=\mathcal{H}om_G(C_p(\tilde{X}),V_n)$  
is precisely the direct sum of FI$[G]$-modules $V^{\oplus d_p}$. Therefore, finite generation of the FI-module $H^p(X;V)$  follows since it is a subquotient of the finitely generated FI-module $C^{p}(X,V)$.

If $R=\mathbb{Q}$, since the weight of an FI-module does not increase when taking extensions, then we have that $V^{\oplus d_p}$ is finitely generated of weight $\leq m$. Moreover,   $V^{\oplus d_p}$ has stability degree $N$ because $V$ has stability degree $N$. Furthermore, the FI-module $H^p(X;V)$  is obtained from the complex of FI-modules 
\begin{equation*}
\xymatrix{
C^{p-1}(X,V)\ar[r]^{\delta_{p-1}}& C^{p}(X,V) \ar[r]^{\delta_p}& C^{p+1}(X,V)}
\end{equation*}
where we have that each FI-module has stability degree $N$ and is finitely generated of degree $\leq m$. The weight of an FI-module is preserved under subquotients and from \cite[Proposition 2.45]{3AMIGOS} applied to the previous complex we get the desired stability degree.
\end{proof}\bigskip

\noindent{\bf Remark:} For each integer $p\geq 0$ we have a covariant functor $H^p(G;\_\_)$ from  {\bf G-Mod$_R$} to {\bf Mod$_R$} (see \cite{BROWN}). Hence, if $V$ is an $\FIG$-module, we have the FI-module $H^p(G;V)$ given by $H^p(G;V)_n :=H^p(G;V_n)$. If $G$ is a group of type $FP_\infty$, then the space $X=K(G,1)$ satisfies the hypotheses of Theorem \ref{FIG} and the FI-module $H^p(X;V)$ is precisely $H^p(G;V)$.\bigskip

\noindent{\bf The case of $\FIsG$-modules:} Next we state the equivalent result to Theorem \ref{FIG} when we take coefficients in a finitely generated  $\FIsG$-module.

\begin{theo}[Cohomology with coefficients in a f.g. $\FIsG$-module] \label{FIsharpG} Let $k$ be any field or $\mathbb{Z}$ and suppose that $G$ is the fundamental group of a connected CW complex $X$ with finitely many cells in each dimension. If $V$ is an $\FIsG$-module over $k$ finitely generated in degree $\leq m$, then  for every $p\geq 0$, the FI\#-module $H^p(X;V)$ is finitely generated in degree $\leq m$.
\end{theo}
\begin{proof}
Clearly $H^p(X;V)$ is a covariant functor from {\bf FI\#} to {\bf Mod}$_k$. First suppose that $k$ is any field. Keeping the notation from the previous proof we have that
$$\dim_k  C^p(X,V_n) =\dim_k V_n^{\oplus d_p}=O(n^m).$$
By hypothesis and \cite[Corollary 2.27]{3AMIGOS} it follows that  $\dim_k V_n=O(n^m)$. Then dimension over $k$ of the subquotient $H^p(X,V_n)$ is $O(n^m)$ and \cite[Corollary 2.27]{3AMIGOS} gives us the desired conclusion. A similar argument applies for $k=\mathbb{Z}$ considering rank instead of dimension.
\end{proof}

\subsection{$\FIG$-modules and spectral sequences}\label{FIBSPEC}

Let $X$ be a connected CW complex with finitely many cells in each dimension and let $x\in X$ be a fixed base point. Suppose that the  fundamental group $\pi_1(X,x)$ is $G$. Consider a functor from {\bf FI$^{op}$} to the category {\bf Fib}($X$) of fibrations over $X$ (a {\it co-FI-fibration over $X$}). Let 
$$E_n\rightarrow X$$ be the fibration associated to {\bf n}, and $H_n$ the fiber over the basepoint $x$. We denote by $E$ the co-FI space of total spaces {\bf n}$\mapsto E_n$ and by $H$ the co-FI space of fibers  {\bf n}$\mapsto H_n$. We can think of $E\rightarrow X$ as a pointwise fibration over $X$ with ``fiber" $H$.\bigskip

 Let us take, together for all $n$, the Leray-Serre spectral sequences associated to each fibration $E_n\rightarrow X$. The functoriality of the Leray-Serre spectral sequence implies that we have a spectral sequence of FI-modules $$E^{p,q}_*=E^{p,q}_*\big(E\rightarrow X\big)$$ converging to the graded FI-module $H^*(E)$.

The $E_2$-page of this spectral sequence is the $\FIG$-module
 $$E_2^{p,q}=H^p\big(X;H^q(H)\big)$$

\noindent {\bf Remark: } Observe that for any $q\geq 0$ and $n\geq 1$, we get an action of the fundamental group $G$ on $H^q(H_n;\mathbb{Q})$ from the $n$-th fibration, which gives to the FI-module $H^q(H;R)$ the structure of an $\FIG$-module over $R$. \bigskip

With this setting, we want to use Theorem \ref{FIG} and the spectral argument given in Section \ref{SS} to determine finiteness conditions for the graded FI-module $H^*(E;R)$ given that we know that the $\FIG$-module $H^q(H;R)$ is finitely generated  over $R$ and we have upper bounds for its degree and its stability degree when $R=\mathbb{Q}$.\bigskip

The typical situation that we will have in the examples in Section \ref{EG2} below is that the $\FIG$-module $H^q(H;\mathbb{Q})$ is finitely generated of weight $\leq \beta q$ with stability degree $\leq \beta q$, for some positive constant $\beta$.  Then Theorem \ref{FIG} gives us the following information about the $E_2$-page.

\begin{lemma}\label{E2}  Suppose that for any $q\geq 0$ the $\FIG$-module $H^q(H;\mathbb{Q})$ is finitely generated of weight $\leq \beta q$ with stability degree $\leq \beta q$. Then, for any $p,q\geq0$, the FI-module  $E_2^{p,q}=H^p\big(X;H^q(H)\big)$ has weight $\leq \beta q$ and stability degree $\leq \beta q$.
\end{lemma}

For a given $i\geq 0$ and $0\leq p\leq i$, the FI-module $E_{\infty}^{p,i-p}$ is a subquotient of $E_{2}^{p,i-p}$. Since the weight of an FI-module cannot increase when taking subquotients, it follows that the FI-module $E_{\infty}^{p,i-p}$ is finitely generated of weight $\leq \beta i$. Moreover, the spectral sequence gives a natural filtration of $H^i(E)$ by FI-modules

$$0\subseteq F^i_i\subseteq F^i_{i-1}\subseteq \ldots\subseteq F^i_1\subseteq F^i_0=H^i(E),$$  where, for $0\leq p\leq i$, the FI-module  $F^i_p$ is an extension of $F^i_{p+1}$ by $E_{\infty}^{p,i-p}$ of weight $\leq \beta i$. Since, by definition, the weight of an FI-module is preserved under extensions, therefore $H^i(E)$ has weight at most $\beta i$.
 
Furthermore, we have precisely the setting described in Section \ref{SS} for constants $\alpha=0$ and $\beta>0$  and Theorem \ref{MAINSPEC} takes the following form.

\begin{theo}\label{LERAYSERRE} For any $i\geq 0$ the FI-module $H^i(E;\mathbb{Q})$ is  finitely generated of weight at most $\beta i$ and has  stability type at most $(2\beta i, \beta i)$.
\end{theo}

\noindent {\bf The case of group extensions:} 
Let $G$ be a group of type $FP_\infty$. Consider a functor from {\bf FI$^{op}$} to the category of group extensions with quotient $G$ and isomorphisms of such (a {\it co-FI-group extension of $G$}).  Let 
$$ 1\rightarrow H_n\rightarrow E_n\rightarrow G\rightarrow 1$$
be the group extension associated to {\bf n} and denote by $E$ and $H$ the corresponding co-FI groups {\bf n}$\mapsto E_n$ and by {\bf n}$\mapsto H_n$.  
For each group extension there is an associated fibration
$$K(E_n,1)\rightarrow K(G,1)$$
with fiber over a fixed base point $x\in K(G,1)$ an Eilenberg-Maclane space $K(H_n,1)$.
Observe that the space $K(G,1)$ has the homotopy type of  a connected CW complex with finitely many cells in each dimension since $G$ is of type $FP_\infty$.  Hence, this gives us a functor from {\bf FI$^{op}$} to {\bf Fib}($K(G,1)$) as in the setting of Section \ref{FIBSPEC} and we obtain the conclusion of Theorem \ref{LERAYSERRE} about  the FI-modules $H^i(E)$. \bigskip 

\noindent{\bf Remarks:} The Leray-Serre spectral sequence associated to the  fibration above corresponds to the Hochschild-Serre spectral sequence associated to the original group extension. Notice that we could have considered this spectral sequence in our previous discussion.\bigskip

\noindent{\bf The Hochschild-Serre spectral sequence and  FI\#-modules:}
Assume that we have a functor from {\bf FI\#$^{op}$} to the category of group extensions with quotient $G$, and not just from {\bf FI$^{op}$} as before. By taking the Hochschild-Serre spectral sequence associated to each group extension with coefficients in any field $k$ or $\mathbb{Z}$, we obtain with a first quadrant spectral sequence of {\bf FI\#}-modules converging to the graded {\bf FI\#}-modules $H^*(E;k)$. Furthermore, suppose that for any $q\geq 0$ the FI\#-module $H^q(H;k)$ is finitely generated over  $k$ in degree $\leq\beta q$, for some $\beta>0$. Then  Theorem \ref{FIsharpG} implies that for any $p,q\geq 0$, the FI\#-module $E_2^{p,q}$ is finitely generated in degree $\leq\beta q$. Since we have the setting from Section \ref{SSFIsharp} with $\alpha=0$ and $\beta>0$, we can conclude that for any $i\geq 0$ the FI\#-module $H^i(E,k)$ is finitely generated in degree $\leq \beta i$.

\section{Sequences of cohomology groups as  FI-modules (part II)}\label{EG2}

Let us apply the perspective described in Section \ref{FIGMod} to understand other sequences of cohomology groups as finitely generated FI-modules. Most of these sequences were already considered in \cite{JIM}. We will see here how the FI-module approach allows us to obtain more information.

\subsection{Cohomology of  moduli spaces $\mathcal{M}_{g,n}$ and pure mapping class groups\\ of surfaces}\label{PURESUR}
Let $2g+r>2$ and consider the functor from {\bf FI$^{op}$} to the category of group extensions of $G=\Mod_{g,r}$ given as follows. The group extension associated to $\mathbf{n}$ is
$$1\rightarrow \pi_1(\C_n(\Sigma_g^{r}))\rightarrow \PMod^n(\Sigma_{g,r})\rightarrow \Mod_{g,r}\rightarrow 1.$$
This is the {\it Birman exact sequence} introduced by Birman in \cite{BIRMAN}. A proof of the exactness can be found in \cite{FARBMARG}. To see that this association is indeed functorial we refer the reader to \cite[Section 5] {JIM}.

From \cite[Proposition 4.1]{4AMIGOS} we have that $H^q\big(\pi_1[\C_n(\Sigma_g^{r})];k\big)=H^q\big(\C_n(\Sigma_g^{r});k\big)$ is a finitely generated  $\FIG$-module over any field $k$. When $k=\mathbb{Q}$, it follows from Theorem \ref{DIM2} that it has weight $\leq 2q$  and stability degree $\leq 2q$. From our discussion in Section \ref{FIBSPEC} with $\beta=2$, we obtain the following statement.

\begin{theo}\label{MCG} Let $k$ be a field. 
 For any $i\geq 0$ and  $2g+r>2$ the FI-module  $H^i\big(\PMod^\bullet_{g,r};k\big)$ is finitely generated over $k$. If $k=\mathbb{Q}$, it has weight $\leq 2i$ and stability type at most $(4i,2i)$.\end{theo}

\subsection{Cohomology of pure mapping class groups for higher dimensional manifolds }\label{DIM>3}

Let $M$ be a smooth connected manifold of dimension $d\geq 3$ and suppose that the fundamental group $\pi_1(M)$ has trivial center or $\Diff (M)$ is simply connected. Moreover, we assume that $\Mod(M)$ is of type $FP_{\infty}$.

Consider the functor from {\bf FI$^{op}$} that associates to each {\bf n} the group extensions of $G=\Mod(M)$

\begin{equation}\label{BIRDIM>3}
\xymatrix{
1\ar[r]& \pi_1(\C_n(\mathring{M}))\ar[r] &\PMod^n (M)\ar[r] &\Mod (M) \ar[r] &1,}
\end{equation}

where $\mathring{M}$ denotes the interior of $M$.
For a proof of the existence of this Birman exact sequence see \cite[Section 6]{JIM}. 

Let $\mathfrak{p}=(p_1,\ldots, p_n)\in \C_n(\mathring{M})$  be a fixed base point. Since $d\geq 3$, then from \cite[Theorem 1]{BIRMAN1969} it follows that the fundamental group
$$\pi_1(\C_n(\mathring{M}),\mathfrak{p})\approx\pi_1(\C_n(M),\mathfrak{p})\approx \pi_1(M^n,\mathfrak{p})\approx \prod_{i=1}^n \pi_1(M,p_i).$$ 

 Hence the FI-module $H^q\big(\pi_1(\C_\bullet(\mathring{M});k)\big)$ is precisely $H^q\big( \pi_1(M)^\bullet;k\big)$. If the group $\pi_1(M)$ is of type $FP_{\infty}$, then from Proposition \ref{PROD} we have that this $\FIG$-module is finitely generated over $k$ and has weight $\leq q$  with stability degree $\leq q$ when $k=\mathbb{Q}$. From our discussion in Section \ref{FIBSPEC} with $\beta=1$ we can conclude the following result.

\begin{theo}\label{PUREDIM>3}
Let $M$ be a smooth connected manifold of dimension $d\geq 3$ such that $\pi_1(M)$ is of type $FP_{\infty}$ (e.g. $M$ compact). Suppose that $\pi_1(M)$ has trivial center or that $\Diff(M)$ is simply connected and assume that the group $\Mod(M)$ is of type $FP_{\infty}$. Then for any field $k$ and $i\geq 0$ the FI-module $H^i\big(\PMod^\bullet(M);k\big)$ is finitely generated over $k$ and has weight $\leq i$ and stability type at most $(2i,i)$ when $k=\mathbb{Q}$.
\end{theo}

\subsection{The case of manifolds with boundary}\label{BDRYFIsharp}

When the surface $\Sigma_{g,r}$ in Section \ref{PURESUR} or the manifold $M$ in Section \ref{DIM>3} has nonempty boundary, the cohomology of the corresponding pure mapping class groups actually has an FI\#-module structure.

\begin{prop}\label{BDRY}
Let $k$ be a field or $\mathbb{Z}$ and $i\geq 0$. If $M$ is a connected smooth manifold of dimension $d\geq 2$ with nonempty boundary, then the FI-module  $H^i\big(\PMod^\bullet(M);k \big)$  has the structure of an FI\#-module. In particular,  $H^i\big(\PMod^\bullet(M)\big)$ has injectivity degree $0$ (when $k=\mathbb{Q}$).
\end{prop}
\begin{proof}

We just prove that $\PMod^\bullet(M)$ has the structure of an FI\#-group.
Consider $(A,B,\psi)\in\text{Hom}_{\text{FI\#}}(\mathbf{m},\mathbf{n})$ where $A\subset [m]$, $B\subset [n]$ and $\psi: A\rightarrow B$ is a bijection. In particular the orders $|A|=|B|$. The corresponding morphism
$$(A,B,\psi)_*:\PMod^m(M)\rightarrow\PMod^n(M)$$
is induced from the following composition:

\begin{equation*}
\xymatrix{
\PDiff^m(M)\ar[r]^{|_A}&\PDiff^{|A|}(M)\ar[r]^{\psi}&\PDiff^{|B|}(M)\ar[r]^{\Upsilon\circ c_\varphi}&\PDiff^n(M).}
\end{equation*}

The map $|_A:\PDiff^{\mathfrak{p}}\rightarrow \PDiff^{\mathfrak{p}_A}(M)$  is given by restricting the configuration $\mathfrak{p}\in\C_m(M)$ to the configuration $\mathfrak{p}_A:A\hookrightarrow M$ in $\C_{|A|}(M)$. 

Abusing notation, $\psi:\PDiff^{|A|}(M)\rightarrow \PDiff^{|B|}(M)$ corresponds to the isomporphism  $\PDiff^{\mathfrak{p}_A}\approx\PDiff^{\mathfrak{p}_A\circ\psi^{-1}}$ induced by taking $\mathfrak{p}_A:A\hookrightarrow M$ to  the embedding $\mathfrak{p}_A\circ\psi^{-1}:B\hookrightarrow M$ in $\C_{|B|}(M)$,  using the bijection $\psi:A\rightarrow B$.

Let $R$ be a collar neighborhood of one component of $\partial M$ and fix a diffeomorphism $\varphi: M\rightarrow M\backslash R$. Then, conjugation by $\varphi$ gives us the identification $c_{\varphi}:\PDiff^{|B|}(M)\approx\PDiff^{|B|}(M\setminus R)$.
Finally, we can extend any diffeomorphism $h\in\Diff\big((M\setminus R)\text{ rel }\partial{(M\setminus R)}\big)$ to a diffeomorphism $\Upsilon(h)\in\Diff(M\text{ rel }\partial{M})$ by letting $\Upsilon(h)=h$ in $M\setminus R$ and $\Upsilon(h)$ be the identity in $R$. Therefore we obtain a group homomorphism $$\Upsilon: \PDiff^\mathfrak{q}(M\setminus R)\rightarrow\PDiff^{\Psi_{B,[n]}(\mathfrak{q})}(M),$$

that takes any diffeomorphism $h$ that fixes the configuration $\mathfrak{q}:B\hookrightarrow (M\setminus R)$ in $\C_{|B|}(M\setminus R)$ to a diffeomorphism $\Upsilon(h)$ of $M$ that fixes the configuration $\Psi_{B,[n]}(\mathfrak{q}):[n]\hookrightarrow M$ in $\C_n(M)$ as defined in \cite[Proof of Proposition 4.1]{3AMIGOS}. 
\end{proof}\medskip


When a Birman sequence exists (hypothesis of Theorems \ref{MCG} and \ref{PUREDIM>3}), we do have a co-FI\#-group extension of $\Mod(M)$.  
Moreover, Theorem \ref{CONF_BDRY} states the finite generation of the cohomology of configuration spaces of manifolds with boundary.
 Hence, the argument at the end of Section \ref{FIGMod} implies the following results.

\begin{theo}\label{MCGsharp} Let $k$ be any field or $\mathbb{Z}$.
 For any $i\geq 0$, $2g+r>2$  and $r>0$ the FI-module  $H^i\big(\PMod^\bullet_{g,r};k\big)$ has the structure of an FI\#-module which is finitely generated in degree $\leq 2i$.\end{theo}

\begin{theo}\label{PUREDIM>3sharp} Let $k$ be any field or $\mathbb{Z}$.
Let $M$ be a smooth connected manifold of dimension $d\geq 3$ with non-empty boundary that satisfies the hypotheses  of Theorem \ref{PUREDIM>3}. Then, for any $i\geq 0$, the FI-module $H^i\big(\PMod^\bullet(M);k\big)$ has the structure of an FI\#-module that is finitely generated in degree $\leq 2i$.
\end{theo}

From the classification of FI\#-modules given in \cite[Theorem 2.24]{3AMIGOS} and the cases when $k$ is either $\mathbb{Z}$ or the fields $\mathbb{Q}$ or $\mathbb{Z}/p\mathbb{Z}$ in Theorems \ref{MCGsharp} and \ref{PUREDIM>3sharp}, we obtain Theorems \ref{BETTI} and \ref{BETTI2}, respectively.

\subsection{Cohomology of classifying spaces for some diffeomorphism groups}
Let $M$ be a connected and compact smooth manifold of dimension $d\geq 3$. We have a fiber bundle
\begin{equation}\label{FIB}
\BPDiff^n(M)\rightarrow \BDiff(M\text{ rel }\partial M)
\end{equation}
\noindent where the ``fiber'' is given by $\Diff(M\text{ rel  }\partial M)/\PDiff^n(M) \approx \C_n(\mathring{M})$, the configuration space of $n$ ordered points in $\mathring{M}$, the interior of $M$. This gives us a functor from {\bf FI$^{op}$} to the category {\bf Fib}$\big(\BDiff(M\text{ rel }\partial M)\big)$.
The hypotheses in the Theorem below give the setting needed to apply the arguments in Section \ref{FIBSPEC} with $\beta=1$ to get the desired conclusion.

\begin{theo}\label{DIFFDIM>3} Let $M$ be a connected real manifold of dimension $d\geq 3$. Suppose that the classifying space $B\Diff(M\text{ rel }\partial M)$ has the homotopy type of a CW-complex with finitely many cells in each dimension. Then, for any field $k$ and $i\geq 0$, the  FI-module 
$ H^i(B\PDiff^\bullet M;k)$ is finitely generated over $k$ and has weight $\leq i$ and stability type at most  $(2i,i)$ when $k=\mathbb{Q}$.
\end{theo}









\section{Application to cohomology of some wreath products}\label{WREATH}

Let $G$ be a group of type $FP_\infty$. 
The wreath product $G\wr S_n$ is the semidirect product $ G^n\rtimes S_n$, where $S_n$ acts on $G^n$ by permuting the coordinates. Therefore there is a split short exact sequence
$$1\rightarrow G^n \rightarrow G \wr S_n\rightarrow  S_n\rightarrow 1$$ 

For any $i\geq 0$ and any partition $\lambda$, a transfer argument implies that the dimension of $H^i\big(G\wr S_n;V(\lambda)_n\big)$ is equal to the multiplicity of $V(\lambda)_n$ in $H^i\big(G^n;\mathbb{Q}\big)$. But, from Proposition \ref{PROD} and \cite[Proposition 2.58]{3AMIGOS}, this multiplicity is constant for $n\geq 2i$ . Hence we obtain cohomological stability for the group $G\wr S_n$ with coefficients in any $S_n$-representation for any $n\geq 2i$.

More generally, let $PK$ be a co-FI-group given by $\mathbf{n}\mapsto PK_n$. Assume that there is a sequence of groups $K_n$ such that, for each $n$, we have the following short exact sequence:

$$1\rightarrow PK_n \rightarrow K_n\rightarrow  S_n\rightarrow 1$$

The wreath product $G\wr K_n$ is the semidirect product $ G^n\rtimes K_n$, where $K_n$ acts on $G^n$ via the surjection $K_n\rightarrow S_n$. Therefore there is a split short exact sequence
$$1\rightarrow G^n\times PK_n \rightarrow G \wr K_n\rightarrow  S_n\rightarrow 1$$

On the other hand, for any $i\geq 0$, the naturality of the K\"unneth formula implies the following isomorphism of FI-modules:

$$H^i(G^{\bullet}\times PK)=\bigoplus _{p+q=i}H^p(G^{\bullet})\otimes H^q(PK).$$

Suppose that the graded FI-module $H^*(PK)$ is known to be of finite type. In \cite[Proposition 2.61]{3AMIGOS} is proved that finite generation is closed under tensor products, therefore the FI-modules  $H^p(G^{\bullet})\otimes H^q(PK)$ are finitely generated  for $p,q\geq 0$ such that $p+q=i$. Moreover, $$\text{weight}\big(H^p(G^{\bullet})\otimes H^q(PK)\big)\leq \text{weight}\big(H^p(G^{\bullet})\big)+\text{weight}\big(H^p(PK)\big).$$ It follows that the consistent sequence $H^i(G^{n}\times PK_n;\mathbb{Q})$ is monotone and uniformly representation stable (although we do not always get a specific stable range).

As before, the dimension of  $H^i\big(G\wr K_n;V(\lambda)_n\big)$  is given by the multiplicity of $V(\lambda)_n$ in $H^i\big(G^n\times PK_n;\mathbb{Q}\big)$, which is eventually constant by uniform representation stability. Therefore we have that  $H^i\big(G\wr K_n;V(\lambda)_n\big)\approx H^i\big(G\wr K_{n+1};V(\lambda)_n\big)$ for any $n$ sufficiently large. In particular, we obtain rational homological stability for the groups $G\wr K_n$.


Parts (ii), (iii) and (iv) of Theorem \ref{COHOWREATH} follow from applying the above discussion to the short exact sequences:
$$1\rightarrow P_n(\Sigma_{g,r})\rightarrow B_n(\Sigma_{g,r})\rightarrow  S_n\rightarrow 1$$
$$1\rightarrow \PMod_{g,r}^n \rightarrow \Mod_{g,r}^n\rightarrow  S_n\rightarrow 1$$
$$1\rightarrow \PMod^n(M) \rightarrow \Mod^n(M)\rightarrow  S_n\rightarrow 1$$

To obtain Theorem \ref{COHOWREATH}  part (v), we consider the co-FI-groups $\PSigma_\bullet$  and $\Sigma_\bullet^+$, which are functors from {\bf FI$^{op}$} to {\bf Gp} given by $\mathbf{n}\mapsto \PSigma_n$, the {\it pure string motion group} (see definition in \cite[Sections 1 \& 2]{WILSON})  and $\mathbf{n}\mapsto \Sigma_n^+$, the braid permutation group, respectively.   
In \cite[Theorem 6.4]{WILSON} Wilson proved that for any $k\geq 0$ the sequence $\big\{H^k(\PSigma_n;\mathbb{Q})\big\}$ satisfies uniform representation stability with stable range $n\geq 4k$.  Therefore, \cite[Theorem 1.14]{3AMIGOS} implies that for any $k\geq 0$ the FI-module $H^k(\PSigma_\bullet)$  is finitely generated.

The co-FI-groups  $\Sigma_\bullet^+$ and $\PSigma_\bullet$ are related by the following short exact sequence:

$$1\rightarrow P\Sigma_n \rightarrow \Sigma_n^+\rightarrow  S_n\rightarrow 1,$$

which give us again the setting discussed above.

\bibliographystyle{amsalpha}
\bibliography{referFI}

\providecommand{\bysame}{\leavevmode\hbox to3em{\hrulefill}\thinspace}
\providecommand{\MR}{\relax\ifhmode\unskip\space\fi MR }
\providecommand{\MRhref}[2]{%
  \href{http://www.ams.org/mathscinet-getitem?mr=#1}{#2}
}
\providecommand{\href}[2]{#2}
\begin{thebibliography}{CEFN12}

\bibitem[BCT89]{BOD_COHEN_TAYLOR}
C.-F. Bodigheimer, F.~Cohen, and L.~Taylor, \emph{On the homology of
  configuration spaces}, Topology (1989), no.~28, 111--123.

\bibitem[Bir69]{BIRMAN1969}
J.~S. Birman, \emph{On braid groups}, Comm. Pure Appl. Math. \textbf{22}
  (1969), 41--72.

\bibitem[Bir74]{BIRMAN}
\bysame, \emph{Braids, links, and mapping class groups}, Annals of Mathematics
  Studies, no.~82, Princeton University Press, Princeton, N.J., 1974.

\bibitem[Bro94]{BROWN}
K.~S. Brown, \emph{Cohomology of groups}, Graduate Texts in Mathematics,
  vol.~87, Springer-Verlag, New York, 1994, Corrected reprint of the 1982
  original.

\bibitem[CEF12]{3AMIGOS}
T.~Church, J.~Ellenberg, and B.~Farb, \emph{$\text{FI}$-modules: a new approach
  for $\text{S}_n$-representations},
  \href{http://http://arxiv.org/abs/1204.4533}{arXiv:1204.4533}.

\bibitem[CEF13]{3AMIGOS2}
\bysame, \emph{Representation stability in cohomology and asymptotics of
  families of varieties over finite fields}, Preprint.

\bibitem[CEFN12]{4AMIGOS}
T.~Church, J.~Ellenberg, B.~Farb, and R.~Nagpal, \emph{$\text{FI}$-modules over
  $\text{N}$oetherian rings}, arXiv:1210.1854.

\bibitem[CF]{CHURCH_FARB}
T.~Church and B.~Farb, \emph{Representation theory and homological stability},
  \href{http://arxiv.org/abs/1008.1368}{arXiv:1008.1368}.

\bibitem[Chu]{CHURCH}
T.~Church, \emph{Homological stability for configuration spaces of manifolds},
  to appear in Invent. Math.

\bibitem[CT93]{COHEN_TAYLOR}
F.~R. Cohen and L.~R. Taylor, \emph{On the representation theory associated to
  the cohomology of configuration spaces}, Algebraic Topology, Contemp. Math.
  (1993), no.~146, 91--109.

\bibitem[Die87]{DIECK}
T.~t. Dieck, \emph{Transformation groups}, De Gruyter Studies in Mathematics,
  vol.~8, Walter de Gruyter, Berlin; New York, 1987.

\bibitem[FH91]{FULTON_HARRIS}
W.~Fulton and J.~Harris, \emph{Representation theory: A first course}, Graduate
  Texts in Mathematics, vol. 129, Springer-Verlag, New York, 1991, Readings in
  Mathematics.

\bibitem[FM12]{FARBMARG}
B.~Farb and D.~Margalit, \emph{A primer on mapping class groups}, Princeton
  Mathematical Series, vol.~49, Princeton University Press, Princeton, NJ,
  2012.

\bibitem[FP]{FABERPAN}
C.~Faber and R.~Pandharipande, \emph{Tautological and non-tautological
  cohomology of the moduli space of curves}, to appear in the Handbook of
  Moduli.

\bibitem[FT05]{FELIX}
Y.~F\'{e}lix and D.~Tanr\'{e}, \emph{The cohomology algebra of unordered
  configuration spaces}, J. London Math. Soc. (2005), no.~72 (2).

\bibitem[Har85]{HARER}
J.~L. Harer, \emph{Stability of the homology of the mapping class groups of
  orientable surfaces}, Ann. of Math. (2) \textbf{121} (1985), no.~2, 215--249.

\bibitem[Har88]{HARERmoduli}
\bysame, \emph{The cohomology of the moduli space of curves}, Theory of moduli
  (Montecatini Terme, 1985), Lecture Notes in Math., vol. 1337, Springer,
  Berlin, 1988, pp.~138--221.

\bibitem[Hat02]{HATCHER}
A.~Hatcher, \emph{Algebraic topology}, Cambridge University Press, Cambridge,
  2002.

\bibitem[HL97]{HAINLOO}
R.~Hain and E.~Looijenga, \emph{Mapping class groups and moduli spaces of
  curves}, Algebraic geometry---{S}anta {C}ruz 1995, Proc. Sympos. Pure Math.,
  vol.~62, Amer. Math. Soc., Providence, RI, 1997, pp.~97--142.

\bibitem[HM97]{HATCHER_MC}
A.~Hatcher and D.~McCullough, \emph{Finiteness of classifying spaces of
  relative diffeomorphism groups of {$3$}-manifolds}, Geom. Topol. \textbf{1}
  (1997), 91--109.

\bibitem[HW10]{HATCHER_WAHL}
A.~Hatcher and N.~Wahl, \emph{Stabilization for mapping class groups of
  3-manifolds}, Duke Math. J. \textbf{155} (2010), no.~2, 205--269.

\bibitem[JR11]{JIM}
R.~Jimenez~Rolland, \emph{Representation stability for the cohomology of the
  moduli space $\mathcal{M}_g^n$}, Algebraic \& Geometric Topology \textbf{11}
  (2011), no.~5, 3011--3041.

\bibitem[RW]{OSCAR}
O.~Randal-Williams, \emph{Homological stability for unordered configuration
  spaces}, to appear in the Quarterly Journal of Mathematics.

\bibitem[Tot96]{TOTARO}
B.~Totaro, \emph{Configuration spaces of algebraic varieties}, Topology
  \textbf{35} (1996), no.~4, 1057--1067.

\bibitem[Wah12]{WAHL}
N.~Wahl, \emph{Homological stability for mapping class groups of surfaces},
  Handbook of Moduli \textbf{III} (2012), Advanced Lectures in Mathematics 26.
  547--583.

\bibitem[Wil12]{WILSON}
J.~Wilson, \emph{Representation stability for the cohomology of the pure string
  motion groups}, Algebraic \& Geometric Topology \textbf{12} (2012), no.~2,
  909--931.

\end{thebibliography}

\begin{small}
\noindent Department of Mathematics\\
University of Chicago\\
5734 University Ave.\\
Chicago, IL 60637\\
E-mail: \texttt{\href{mailto:atir83@math.uchicago.edu}{atir83@math.uchicago.edu}}
\end{small}
 \end{document}